\documentclass{amsart}

\usepackage[english]{babel}

\usepackage[letterpaper,top=2cm,bottom=2cm,left=3cm,right=3cm,marginparwidth=1.75cm]{geometry}

\usepackage{mathrsfs}  
\usepackage{amsmath}
\usepackage{graphicx}
\usepackage[colorlinks=true, allcolors=blue]{hyperref}
\usepackage{amsthm}

\newtheorem{thm}[equation]{Theorem}
\newtheorem{cor}[equation]{Corollary}
\newtheorem{lemma}[equation]{Lemma}

\theoremstyle{definition}
\newtheorem{remark}[equation]{Remark}

\newtheorem{question}[equation]{Question}

\newcommand{\nc}[2]{\newcommand{#1}{#2}}
\nc{\on}{\operatorname}

\title{The Manin-Mumford conjecture in genus 2 and rational curves on K3 surfaces}
\author{Philip Engel, Raju Krishnamoorthy, Daniel Litt}

\begin{document}
\maketitle

\begin{abstract}
Let $A$ be a simple abelian surface over an algebraically closed field $k$. Let $S\subset A(k)$ be the set of torsion points $x$ of $A$ such that there exists a genus $2$ curve $C$ and a map $f: C\to A$ such that $x$ is in the image of $f$, and $f$ sends a Weierstrass point of $C$ to the origin of $A$. The purpose of this note is to show that if $k$ has characteristic zero, then $S$ is finite --- this is in contrast to the situation where $k$ is the algebraic closure of a finite field, where $S=A(k)$, as shown by Bogomolov and Tschinkel. We deduce that if $k=\overline{\mathbb{Q}}$, the Kummer surface associated to $A$ has infinitely many $k$-points not contained in a rational curve arising from a genus $2$ curve in $A$, again in contrast to the situation over the algebraic closure of a finite field.
\end{abstract}

\section{Introduction}

Bogomolov and Tschinkel have observed that for a Kummer K3 surface $X$ over an algebraic closure of a finite field $\overline{\mathbb{F}_q}$, there exists a  rational curve passing through every $\overline{\mathbb{F}_q}$-point of $X$ \cite[Theorem 1.1]{BT:finite-fields}. Equivalently, for an abelian surface $A/\overline{\mathbb{F}_q}$, any point $x\in A(\overline{\mathbb{F}_q})$ lies in the image of a morphism $f:C\to A$ from a
hyperelliptic curve $C$ that sends a Weierstrass point $O_C$ to the origin $O_A$.
They point out that as yet we have not ruled out a similar phenomenon over $\overline{\mathbb{Q}}$ \cite{BT:elliptic, BT:finite-fields}. 

The goal of this note is to rule out a slightly weaker phenomenon. Bogomolov and Tschinkel in fact show that every point $x\in A(\overline{\mathbb{F}_q})$ in an abelian surface lies in the image of a map $f: (C,O_C)\to (A,O_A)$ from a \emph{genus $2$ curve $C$}, with $O_C$ a marked Weierstrass point.  In contrast, we show:
\begin{thm}\label{thm:mainthm}
Let $A$ be a simple abelian surface over an algebraically closed field of characteristic zero. The set of torsion points $x\in A$, such that there exists a map $f\colon (C,O_C)\to (A,O_A)$ from a genus $2$ hyperelliptic curve with a marked Weierstrass point,  containing $x$ in its image, is finite. \end{thm}
In principle, the bound obtained is explicit in terms of the Galois representation on the Tate module of a descent of $A$ over a finitely generated field.

The condition that $f$ send a Weierstrass point to the origin is crucial, as the  translates of one image $f(C)$ will cover $A$. Moreover the result is false if one does not assume simplicity; for example, if $A=E\times E$ is the square of a CM elliptic curve, every torsion point is in the image of a map from a genus $2$ curve sending a Weierstrass point to the origin. Finally, note that as the Torelli map in genus $2$ is dominant, there are in general infinitely many isomorphism classes of genus $2$ curves mapping to a given abelian surface. Our theorem thus shows that a countably infinite collection of curves in $A$ contains only finitely many torsion points.

Now let $A$ be an abelian surface, and $X:=\on{Bl}_{A[2]}A/\{\pm 1\}$ the associated Kummer K3 surface. Given a rational curve $C$ in $X$, we let $C'$ be the preimage in $A$. This curve is necessarily irreducible because $A$ contains no rational curves. Let $h(C)$ be the genus of its normalization. As a corollary to Theorem \ref{thm:mainthm}, we have
\begin{cor}\label{cor:main-cor}
Let $X/\overline{\mathbb{Q}}$ be a Kummer K3 surface associated to a simple abelian surface over $\overline{\mathbb{Q}}$. Then there exist infinitely many $\overline{\mathbb{Q}}$-points of $X$ not contained in any rational curve $C$ with $h(C)=2$.
\end{cor}
Indeed, the image in $X$ of any point of $A$ of finite order $N\gg 0$ will not be contained in such a curve $C$, by Theorem \ref{thm:mainthm}.
Again, this is in contrast to Bogomolov and Tschinkel's result over the algebraic closure of a finite field.

The primary antecedent to this result is the main result of \cite{BM:smooth-curves}, which shows that for certain elliptic K3 surfaces over $\overline{\mathbb{Q}}$, there exist $\overline{\mathbb{Q}}$-points not contained in any \emph{smooth} rational curves. In contrast, not all of the rational curves in Corollary \ref{cor:main-cor} (that is, those $C$ with $h(C)=2$) are smooth.

We'd like to gratefully acknowledge Ananth Shankar for useful conversations at the outset of this project. 
\section{Proof of Theorem \ref{thm:mainthm}}
Theorem \ref{thm:mainthm} follows from an analysis of a proof of the Manin-Mumford conjecture, essentially due to Lang \cite{lang:division} and Serre \cite[Letter to Ribet, 1/1/1981]{serre-oeuvres} (see e.g.~\cite[Section 2]{BR:Manin-Mumford} for a discussion of the proof). In particular, the main input is a variant of a result of Serre on homotheties in the image of the Galois action on the Tate module of an abelian surface over a number field, due to Wintenberger.

Let $k$ be a field of characteristic zero and $\overline{k}$ an algebraic closure. For an abelian variety $A/k$, let $T_\ell(A)=\varprojlim_n A[\ell^n](\overline{k})$ be the $\ell$-adic Tate module of $A$, $T(A)=\prod_\ell T_\ell(A)$ the total Tate module of $A$, and $$\rho_{A,\ell}: \text{Gal}(\overline{k}/k)\to GL(T_\ell(A)),\hspace{10pt} \rho_A: \text{Gal}(\overline{k}/k)\to GL(T(A))$$ the associated $\ell$-adic (resp.~adelic) Galois representations. For each $\ell$, let $H_\ell\subset GL(T_\ell(A))$ be the subgroup consisting of $\ell$-adic scalar matrices (i.e.~homotheties), and let $H=\prod_\ell H_\ell\subset GL(T(A))$ be the adelic scalar matrices. Wintenberger shows (verifying a conjecture of Lang in dimension $\leq 4$):
\begin{thm}[Wintenberger, {\cite[Corollaire 1]{wintenberger2002demonstration}}]\label{thm:serre-thm}
Let $A$ be an abelian variety over a finitely generated field $k$ of characteristic zero, with $\dim(A)\leq 4$. Then the intersection of $\rho_{A}(\text{Gal}(\overline{k}/k))$ with $H$ is open. 
\end{thm}
\begin{remark}
Wintenberger only states this result for $k$ a number field, but the result for finitely generated fields of characteristic zero follows as in \cite[Letter to Ribet 1/1/1981, \S 1]{serre-oeuvres}.
\end{remark}
For $A$ as in the theorem, let $C(A)$ denote the index of $$[\rho_{A}(\text{Gal}(\overline{k}/k))\cap H: H].$$ By the theorem, $C(A)$ is finite. In particular, Wintenberger's theorem implies that for $\dim(A)\leq 4$, the set of $\ell$-adic homotheties in the image of Galois is all of $\mathbb{Z}_\ell^*\cdot \on{id}$, for almost all $\ell$. 

Similarly, Serre shows \cite[Letter to Ribet, 1/1/1981]{serre-oeuvres} that for any abelian variety $A$ (with no restriction on the dimension) over a finitely generated field  of characteristic zero, the constants $$C_\ell(A):=[\rho_{A,\ell}(\on{Gal}(\overline{k}/k)\cap H_\ell: H_\ell]$$ are bounded independent of $\ell$.  Let $S(A):=\max_\ell\, C_\ell(A)$ be the maximum of all of these indices. Note that if $C(A)$ is finite, we have $S(A)\leq C(A)$.

\begin{thm}[Manin-Mumford Conjecture, Ribet, {\cite[Section 2]{BR:Manin-Mumford}}]\label{thm:manin-mumford}
Let $X$ be a smooth projective curve of genus $g\geq 2$ over a finitely generated field $k$ of characteristic zero, and let $x\in X(k)$ be a point. Then there exists an integer $N=N(S(\text{Jac}(X)))$ such that if $X$ is embedded in $\text{Jac}(X)$ via the Abel-Jacobi map associated to $x$, then the image of $X(\overline{k})$ in $\text{Jac}(X)(\overline{k})$ does not contain any $m$-torsion points with $m>N$.
\end{thm}

Here $S(\text{Jac}(X))$ is the constant defined above.
Aside from this ``explicit" form of the Manin-Mumford conjecture, the main observation going into the proof of Theorem \ref{thm:mainthm} is the following lemma on the fields of definition of genus $2$ curves mapping to an abelian surface.
\begin{lemma}\label{lem:field-of-defn}
Let $A$ be a geometrically simple abelian surface over a finitely generated field $k$ of characteristic zero. Then there exist constants $C'=C'(A)$ such that if $X$ is a curve of genus $2$ over $\overline{k}$ and $f: X\to A_{\bar k}$ is a non-constant map sending a Weierstrass point of $X$ to the origin, then:
\begin{itemize}
    \item there exists a finite extension $L\subset \overline{k}$ of $k$, a curve $X'/L$, a Weierstrass point $x\in X'(L)$, a map $f': X'\to A_{L}$ sending $x$ to the origin, and an identification of $f'_{\overline{k}}$ with $f$, such that
    \item  The constant  $C(\on{Jac}(X'))$ satisfies $C(\on{Jac}(X'))<C'$.
\end{itemize}
\end{lemma}
\begin{proof}

Since $A$ is simple, the map $f_*: \on{Jac}(X_{\bar k})\to A_{\bar k}$ is an isogeny, and the induced map $T(\on{Jac}(X_{\bar k}))\to T(A)$ has a finite index image $T'\subset T(A)$. Let $G'\subset \on{Gal}(\overline{k}/k)$ be the stabilizer of $T'$ under $\rho_A$ and let $k'$ be the fixed field of $G'$. The Jacobian $\on{Jac}(X_{\bar k})$ and the map $f_*$ descend to an abelian variety $J'$ and a map $g: J'\to A_{k'}$ over $k'$ (as, for example, the kernel of the isogeny dual to $f_*$ is stable under the action of the absolute Galois group of $k'$).

We now descend the principal polarization on $\on{Jac}(X)$ to $J'$, after a field extension of absolutely bounded degree. Let $NS(J')$ be the N\'eron-Severi group of $J'$; this is a free abelian group of rank at most $4$. The action of $\on{Gal}(\overline{k}/k')$ on $NS(J')$ has finite image. Hence the order of this image is bounded above, for example, by the order of the maximal finite subgroup of $GL_4(\mathbb{Z})$, hence is independent of $X$. Let $k''$ be the fixed field of the kernel of this action, so that each component of $\underline{\on{Pic}}(\on{Jac}(X))$ descends to $k''$. After possibly replacing $k''$ by a finite extension of absolutely bounded degree, we may assume in addition that the absolute Galois group of $k''$ acts trivially on the $2$-torsion group $\underline{\on{Pic}}^0(J')_{k''}[2](\overline{k})$.

Choose an embedding of $X$ in $\on{Jac}(X)$ sending a Weierstrass point to the origin, and let $\mathscr{L}:=\mathscr{O}_{\on{Jac}(X)}(X)\in \on{Pic}(\on{Jac}(X))$ be the corresponding line bundle, a.k.a.~the theta bundle. The line bundle $\mathscr{L}$ yields a geometric point $[\mathscr{L}]$ of $\underline{\on{Pic}}(J')_{k''}$, and for any $\sigma\in \on{Gal}(\overline{k''}/k'')$, $[\mathscr{L}]-[\mathscr{L}]^\sigma$ is a $2$-torsion point of $\underline{\on{Pic}}^0(J')_{k''}$ (as the property of sending a Weierstrass point to the origin may be checked after base change to $\overline{k}$, and any two Weierstrass points differ by a $2$-torsion point).

Thus the orbit of $[\mathscr{L}]$ under  $\underline{\on{Pic}}^0(J')_{k''}[2]\simeq(\mathbb{Z}/2)^{4}$ is defined over $k''$; there exists an extension $L$ of $k''$ of absolutely bounded degree splitting this $\underline{\on{Pic}}^0(J')_{k''}[2]$-torsor. (Explicitly, the torsor itself is the spectrum of an \'etale $k''$-algebra of degree $16$, and it will split over the residue field of any of its closed points).

Now the line bundle $\mathscr{L}$ descends to $L$; let $s\in H^0(\mathscr{L})$ be any non-zero global section. As $\mathscr{L}$ is the theta bundle, $h^0(\mathscr{L})=1$, and so $X':=V(s)$ is a descent of $X$ to $L$. Moreover the embedding of $X'$ in its Jacobian given by $\mathscr{L}$ sends a Weierstrass point to the origin by construction; hence $X'$ is equipped with a rational Weierstrass point $x$. We now need only verify that the condition of the second bullet point is satisfied, i.e. we must bound $C(\on{Jac}(X'))$ purely in terms of $C(A)$.

First, note that $C(J')$ is equal to $C(A)$, as any homothety in $GL(T(A))$ stabilizes any subgroup of $T(A)$. That is, the index of the image of the absolute Galois group of $k'$ in the homotheties of $T(A)$ is bounded above by $C(A)$. Now observe that the degrees of $k''$ and $L$ over $k'$ are all absolutely bounded, so we're done.
\end{proof}
We may now prove the main result:
\begin{proof}[Proof of Theorem \ref{thm:mainthm}]
Let $A$ be a simple abelian surface over an algebraically closed field as in the statement of the theorem. $A$ descends to a geometrically simple abelian surface over some finitely generated field $k$ of characteristic zero. Lemma \ref{lem:field-of-defn} gives us a constant $C'=C'(A)$ such that for any genus $2$ curve $X$ with a nonconstant map $f: X\to A_{\overline{\mathbb{Q}}}$ sending a Weierstrass point to the origin, $(X,f)$ is defined over a finitely generated field $k'$ such that $$C_\ell(\on{Jac}(X)):=[\rho_{\text{Jac}(X),  \ell}(\text{Gal}(\overline{k}/k'))\cap H_\ell: H_\ell]$$ is bounded above by $C'$. Thus by Theorem \ref{thm:manin-mumford}, there exists a constant $N(C')$ such that, under an embedding $X\hookrightarrow \text{Jac}(X)$ arising from a Weierstrass point of $X$, any torsion point of $\text{Jac}(X)$ lying in the image of $X$ has order bounded above by $N(C')$. We claim that any torsion point of $A$ in the image of $f$ also has order bounded above by $N(C')$

But by the Albanese property of $\text{Jac}(X)$, the map $f$ factors through a map $\on{Jac}(X)\to A$; this map is necessarily an isogeny as $A$ is geometrically simple. Let $x$ be a torsion point in the image of $f$; its preimages in $\on{Jac}(X)$ are torsion, and at least one lies in the image of the Abel-Jacobi map. Hence its order is bounded by $N(C')$, by the previous paragraph. This completes the proof.
\end{proof}
It is natural to ask if the analogous statement where one allows the curves $C$ to vary over hyperelliptic curves of all genera holds:
\begin{question}
Let $A$ be a simple abelian surface over $\overline{\mathbb{Q}}$. Let $S\subset A(\overline{\mathbb{Q}})$ be the set of torsion points in $A$ in the image of some
map $f: (C,O_C)\to (A,O_A)$ with $C$ hyperelliptic, $O_C$ Weierstrass. Is $S$ finite?
\end{question}
A positive answer would resolve Bogomolov-Tschinkel's question about rational points contained in a rational curve on a K3 surface over $\overline{\mathbb{Q}}$, i.e.~it would show that there exist Kummer K3 surfaces over $\overline{\mathbb{Q}}$ with a $\overline{\mathbb{Q}}$-point not contained in any rational curve.

\bibliographystyle{alpha}
\bibliography{bogomolov}

\end{document}